\documentclass{amsart}

\def\t{\xi}

\def\be{\begin{equation}}
\def\ee{\end{equation}}
\def\inte{\left[0,T \right)}
\def\bse{\begin{subequations}}
\def\ese{\end{subequations}}
\def\u{\mathbf{u}}
\def\ve{\mathbf{v}}
\def\x{\mathbf{x}}
\def\R{\mathbb{R}}
\def\o{\omega}
\def\O{\Omega}
\def\OO{\mathbf{\Omega}}
\def\k{\kappa}

\def\D{\mathbb{D}}

\def\LH{\mathcal{L}_H}

\def\la{\left\langle}
\def\ra{\right\rangle}
\def\lh{\la\la}
\def\rh{\ra\ra_{h}}

\def\PS{\mathbf{\Psi}}

\def\r{\rho}

\def\Rm{\mathcal{R}}

\def\ep{\varepsilon}

\def\F{\mathbf{F}}
\newcommand{\der}[3][]{\frac{\partial^{#1} #2}{\partial #3^{#1}}}
\newcommand{\dermix}[3]{\frac{\partial^{2} #1}{\partial #2 \partial #3}}
\newcommand{\diff}[3][]{\frac{d^{#1} #2}{d #3^{#1}}}

\usepackage{amssymb}
\newtheorem{defin}{Definition}[section]
\newtheorem{lemma}[defin]{Lemma}
\newtheorem{proposition}[defin]{Proposition}
\newtheorem{theorem}[defin]{Theorem}
\newtheorem{corr}[defin]{Corrolary}
\newtheorem{claim}[defin]{Claim}
\newtheorem*{rema}{Remark}
\numberwithin{equation}{section}
\author[]
{Boris Ettinger and Edriss S. Titi}
\thanks{ }
\address{(B.Ettinger) Department of Mathematics \newline \indent University of California \newline \indent Berkeley, California 94720, USA.} \email{ettinger@math.berkeley.edu}

\address{(E.S.Titi) Department of Mathematics and Department
\newline \indent of Mechanical and Aerospace
Engineering,\newline \indent University of California
\newline \indent Irvine, California 92697,USA. \newline \indent
 {\bf Also:} \newline \indent
 Department of Computer Science
and  Applied Mathematics \newline \indent Weizmann Institute of
Science \newline \indent Rehovot, 76100, Israel}
\email{etiti@math.uci.edu} \,\, \email{edriss.titi@weizmann.ac.il}

\date{February 14, 2008}



\title[Helical Euler]{Global Existence and Uniqueness of Weak Solutions
of 3-D Euler Equations with Helical Symmetry in the Absence of Vorticity Stretching}
\usepackage{amsmath,amssymb,amscd}
\begin{document}
\begin{abstract}
We prove uniqueness and existence of the weak solutions of Euler equations with helical symmetry, with initial vorticity in $L^{\infty}$ under "no vorticity stretching" geometric constraint. Our article follows the argument of the seminal work of Yudovich. We adjust the argument to resolve the difficulties which are specific to the helical symmetry.
\end{abstract}
\maketitle
\section{Introduction}
\par Ideal incompressible homogeneous fluid of density $\r_0$ and confined in three-dimensional domain $\D\subseteq \R^3$ is governed by the Euler equations:
\bse
\label{eq:Ee}
\begin{align}
\label{eq:Eedyn}
\der{\u}{t}+(\u\cdot\nabla)\u&=-\frac{1}{\r_0}\nabla p+\mathbf{F},\\
\nabla\cdot\u&=0,
\end{align}
\ese
supplemented with the no-normal flow boundary conditions
\be
\u\cdot\mathbf{n}=0,\quad \text{on }\partial\D,\quad \text{where }\mathbf{n}\text{ is the normal vector to }\partial \D,
\ee
and initial velocity $\u_0(\x)$.
$\u:\D\times\inte\rightarrow \R^3$ is the velocity field, $p:\D\times\inte\rightarrow \R$ is the pressure, determined by the incompressibility condition and $\mathbf{F}:\D\times\inte\rightarrow \R^3$ is the given external body forcing term. We will consider constant density $\r_0=1$.
\par In this article, we will investigate the solutions of equations (\ref{eq:Ee}), which are invariant under a helical symmetry group $G^{\k}$. The group $G^\k$ is a one-parameter group of isometries of $\R^3$
\be
G^\k=\{S_{\rho}:\R^3\rightarrow \R^3|\rho\in \R\}.
\ee
The transformation $S_{\rho}$ ($S$ stands for "screw motion") is defined by:
\be
\label{eq:screw}
S_{\rho}\begin{pmatrix} x \\ y \\ z \end{pmatrix}=\begin{pmatrix} x\cos\r+y\sin\r \\ -x\sin\r+y\cos\r \\ z+\k\r \end{pmatrix},
\ee
where $\k$ is a fixed nonzero constant length scale. In fact, $S_\r$ is a superposition of a simultaneous rotation around the $\hat{z}$-axis with a translation along the $\hat{z}$-axis. The symmetry lines (orbits of $G^\k$) are concentric helices. We call the solutions, and more generally functions, which are invariant under $G^\k$ - "helical". Since the Euler equations in $\R^3$ are invariant under isometries, then (under mild assumptions of uniqueness) solving the Euler equations in a domain which is invariant under helical symmetry with a helical initial condition and a helical body forcing will give rise to a solution, which is helical for the whole interval of time of it's existence.
\par Observe that $S_{2\pi}$ is a translation by $2\pi\k$ in the $\hat{z}$ direction. Therefore, helical symmetry imposes a periodic boundary conditions in the $\hat{z}$ direction. We will also assume that the physical domain $\D\subseteq \R^3$ is bounded in the $\hat{x}$ and $\hat{y}$ directions, thus imposing a no-normal flow. We say that $\D$ is bounded in an (infinite) cylinder along the $\hat{z}$-axis, which has a finite radius. Since $\D$ is invariant under $G^\k$, it's boundary $\partial\D$ can be thought as being the union $\bigcup\limits_{\r\in\R}S_\r C$, where $C$ is a closed planar curve.
\par The difficulty to establish the global regularity for the solutions of the 3D Euler equations can be best appreciated, when one examines the evolution of the vorticity field $\OO=\nabla\wedge\u$, where $\nabla\wedge$ is the curl (rotor) operator. Taking the curl of both sides of equation (\ref{eq:Eedyn}) we get:
\be
\label{eq:Evort}
\der{\OO}{t}+(\u\cdot\nabla)\OO+(\OO\cdot\nabla)\u=\nabla\wedge\mathbf{F}.
\ee
The last term of the left-hand side, $(\OO\cdot\nabla)\u$, is called the vorticity stretching term. This term is the main obstacle to achieve global in time regularity of the three-dimensional Euler Equations, (see \cite{BT},\cite{Cons},\cite{MB} for the latest discussions of this question).  The difficulty remains after imposing the helical symmetry on the solution, since helical flows can undergo nontrivial vorticity stretching.
\par
We will therefore include an additional requirement. We will demand that the velocity field $\u=(u_x,u_y,u_z)^T$, where $u_x,u_y,u_z$ are components of the vector field in the basic directions, obeys the following constraint:
\be
\label{eq:ortas}
yu_x-xu_y+\k u_z=0.
\ee
This condition is an orthogonality of the velocity field to the symmetry lines of the group $G^\k$. This condition together with the assumption of the helical symmetry lead to vanishing of the vorticity stretching term.
We will prove in Section \ref{sc:geo} that under these conditions, the vorticity $\OO$ is directed along the symmetry lines and it's magnitude, up to a normalization is transported by the flow. This control of the $L^{\infty}$ norm of the vorticity is the key to our argument of global existence and uniqueness. It is consistent with the celebrated result of Beal-Kato-Majda \cite{BKM}, and the works of Yudovich \cite{Yud} and Ukhovskii and Yudovich \cite{UYU} for the 2D case and axi-symmetric flow (without swirl), respectively.
\par Our article is inspired by the seminal work of Yudovich \cite{Yud}, who proved existence and uniqueness for a certain class of weak solutions in two space dimensions. We adapt his ideas: the stream function weak formulation of the problem and the elliptic regularity to the case of helical flows. We follow closely his article, especially in Section \ref{sc:uni}, where we prove uniqueness. A different route would have been to extend the ideas of Bardos \cite{Bar}, adding viscosity with artificial boundary conditions and passing to the limit of viscosity going to zero. In such case, the equation with the added viscosity is not the Navier-Stokes equation, because of the different boundary conditions and the limit also has a non-physical boundary conditions.
\par In his work, Yudovich \cite{Yud} proved global existence and uniqueness of the solutions of the two-dimensional Euler equations, whose initial vorticity belongs to the space of essentially bounded functions, $L^{\infty}$. Uniqueness of the solutions was extended for a wider classes of functions in the work Yudovich himself \cite{Yud95} to the class of functions, which are not bounded but whose $L^p$ norms grow "slowly enough". Uniqueness and global existence was also proven by Vishik \cite{Vish} for vorticity in the Besov-like spaces. Global existence of solutions of the two-dimensional Euler equations was proven for initial vorticity $\OO_0\in L^p$ by Majda and DiPerna \cite{MdP} and also for solutions, whose initial vorticity is a positive Radon measure by Delort \cite{Del}. No uniqueness is known in these cases.
\par The helical flows fall within a class of "two-and-a-half" dimensional flows, namely flows in a three-dimensional domains with a certain continuous spatial symmetry. The most heavily investigated within this class are the axi-symmetric flows, which are invariant under a rotation around a certain axis of symmetry. The viscous axi-symmetric flows\footnote{In this article, we will speak always about incompressible flows.} were analyzed by Ladyzhenskaya \cite{Lad}, where the author had to assume that the domain of the flow does not contain a cylinder of a positive radius around the axis of symmetry, in order to prove global well-posedness. The inviscid axi-symmetric flows were the subject of Ukhovski and Yudovich \cite{UYU}, the authors had to assume that the vorticity vanishes rapidly enough near the axis of symmetry, namely $\frac{\o}{r}\in L^{\infty}$ to conclude global existence and uniqueness of the weak solutions. Both of these articles assume that the azimuthal, "swirl" component $u_{\theta}$ of the velocity is zero.
\par The assumption of zero azimuthal components for the axi-symmetric flows means that the vorticity stretching term in (\ref{eq:Evort}) is zero. There were several attempts to analyze the axi-symmetric flows with non-trivial vorticity stretching term. Chae and Immanuilov \cite{ChaIm} proved that solutions exist globally in time and are unique for a "generic" class of solutions. Moreover, Hou and Li \cite{HouLi} constructed recently a family of infinite energy solutions  that can undergo a very rapid and intense vorticity stretching yet do not develop a singularity and exist globally in time.
\par The global in time existence and uniqueness of viscous helical flows for all helical initial data was established by Mahalov et al \cite{MTL} without any restrictions, such as zero vorticity stretching, on the velocity field. Inviscid helical flows were the subject of Dutrifoy \cite{Dut}. The author of \cite{Dut} proves global existence and uniqueness of the classical solutions whose velocity field is orthogonal to the symmetry lines, by essentially using the fact that the vorticity is transported by such a flow and applying the Theorem of Beal, Kato and Majda \cite{BKM}. In this paper, we extend the conclusions of \cite{Dut} to a weaker class of solutions.
\par The helical symmetry introduces a complication, as opposed to planar and axi-symmetrical cases, because one cannot find an orthogonal coordinates, such that one of them runs along the symmetry lines.
Nevertheless, we were able to reduce the Euler equations in our problem to effective equations for two-dimensional velocity on a two-dimensional domain. Using this reduction, we produce an effective vorticity-stream function formulation, which allows us to proceed along the lines of Yudovich's work \cite{Yud}. A similar reduction was proposed in \cite[Section 3]{Dut}, but was not investigated to the full extent of it's consequences.
\par This article is organized as follows. In Section \ref{sc:geo}, we lay down the geometric framework and obtain the consequences of our assumptions to classical solutions of Euler Equations, with this we perform the reduction to the two-dimensional formulation. In Section \ref{sc:fun} we introduce the appropriate functional-analytic setting and formulation of the weak problem.
In Sections  \ref{sc:uni} and \ref{sc:exi}, we prove the uniqueness and the existence of the solutions of the weak problem, respectively.
\section{Geometrical Setting and Stream Function Formulation}
\label{sc:geo}
\par
In this section, we investigate the geometry of helical symmetry and of the orthogonality assumption, Equation \ref{eq:ortas}. We then rewrite the problem of the helical velocity field, which obeys the orthogonality assumption as a problem of two-dimensional velocity field on a two-dimensional domain which obeys a set of effective equations. We introduce the effective vorticity stream function formulation, which will be the basis of the analysis further in the article.
\par A note on our notation. We will usually specify vector and a vector field in a boldface e.g. $\u$ and it's coordinates in a normal font, with indices $x,y,z$ or $1,2,3$ e.g. $\u=(u_x,u_y,u_z)^T$ or $\u=(u_1,u_2,u_3)^T$.
We denote by $(\cdot,\cdot)$ the Euclidean inner product in $\R^3$. It is also used to denote the pointwise inner product of two vector fields.
\subsection{Invariance under helical symmetry}

Denote by $R_\r$ the rotation transformation around the $\hat{z}$ axis by an angle of $\r$ radians:
\be
R_\r=\begin{pmatrix}
\cos\r & \sin\r & 0 \\
-\sin\r & \cos \r & 0 \\
0 & 0 & 1
\end{pmatrix}.
\ee
With this notation, the action of $G^\k$ on $\R^3$ becomes:
\be
S_{\r}(\x)=R_\r(\x)+\begin{pmatrix} 0 \\ 0 \\ \k\r \end{pmatrix}=\begin{pmatrix} x\cos\r+y\sin\r \\ -x\sin\r+y\cos\r \\ z+\k\r \end{pmatrix}.
\ee
We solve our equation in a domain $\D$. Since we look for solutions, which obey boundary conditions imposed by $\D$, and since these solutions have helical symmetry, then $\D$ must also be invariant under a helical symmetry, namely be a union of helices.
\begin{defin}[Helical Domain] A domain $\D\subseteq \R^3$ is helical, if it is invariant under the action of $G^\k$, namely
\be
S_\r\D=\D,\quad \forall \r\in\R.
\ee
\end{defin}
For the rest of this article, we will assume the following assumption about the domain $\D$:
\begin{itemize}
\item $\D$ is helical,
\item $\D$ is simply connected,
\item $\D$ is contained in a cylinder of finite radius around the $\hat{z}$ axis.
\item $\D$ has a twice differentiable boundary.
\end{itemize}
The sole purpose of the simply connectedness assumption is to simplify our proof, it can be removed to yield a more general result.\par
The action of $G^{\k}$ on a scalar function is by composition:
\be
f^{S_\r}(\x)=f(S_{-\r}\x).
\ee
The action of $G^\k$ on vector fields has to take care of the underlying rotation:
\be
\mathbf{v}^{S_\r}(\x)=R_\r \mathbf{v}(S_{-\r}\x).
\ee
The following vector field will have an important role throughout this article:
\be
\label{eq:vect}
\vec{\t}=y \hat{x}-x\hat{y}+\k\hat{z}.
\ee
The vector field $\vec{\t}$ is the field of tangents of the symmetry lines of $G^{\k}$.
\begin{defin}[Helical Function]
\label{df:helf}
A (scalar) function $f:\R^3\rightarrow \R$ is called helical, if
\be
f(S_\r \x)=f(\x),\quad \forall\r\in\R.
\ee
Namely, it is invariant under the action of $G^\k$.
\end{defin}
\begin{claim}
\label{cl:charhelf}
A continuously differentiable function $f:\R^3\rightarrow \R$ is helical if and only if
\be
\label{eq:hesc}
y\der{f}{x}-x\der{f}{y}+\k\der{f}{z}=(\vec{\t},\nabla f)=0.
\ee
\end{claim}
\begin{proof}
Assuming that the function $f$ is helical, differentiate the relation
\be
f(S_\r \x)=f(\x)
\ee
with respect to $\r$. The right-hand side is zero, the left-hand side is
\be
\begin{split}
\diff{}{\r}[f(x\cos\r&+y\sin\r,-x\sin\r+y\cos\r,z+\k\r)]\\
&=(-x\sin\r+y\cos\r)\der{f}{x}+(-x\cos\r-y\sin\r)\der{f}{y}+\k\der{f}{z}.
\end{split}
\ee
Now compute at $\r=0$ to get the desired relation.
\par
To prove the converse of the statement, let $\x\in \R^3$, denote $\mathbf{c}=(c_1,c_2,c_3)^T:\R\rightarrow \R^3$ by $\mathbf{c}(\r)=S_\r\x$. Define a function $g:\R\rightarrow \R$ by:
\be
g(\r)=f(\mathbf{c}(\rho)).
\ee
By virtue of (\ref{eq:hesc}), the function $g(\r)$ satisfies the following differential equation:
\be
\diff{g}{\r}=c_2(\r)\der{f}{x}(\mathbf{c}(\r))-c_1(\r)\der{f}{y}(\mathbf{c}(\r))+\k\der{f}{z}(\mathbf{c}(\r))=0,\quad \text{with } g(0)=f(\x).
\ee
Therefore,
\be
g(\r)=g(0),\text{ and consequently, }f(S_\r \x)=f(\x).
\ee
\end{proof}
We will also use the following notation:
\be
\der{f}{\vec{\t}}=(\vec{\t},\nabla f)=y\der{f}{x}-x\der{f}{y}+\k\der{f}{z}.
\ee
With this notation the last claim states that a differentiable function $f$ is helical if and only if $\der{f}{\vec{\t}}=0$.
\begin{defin}[Helical vector field]
\label{df:helv}
The vector field $\mathbf{v}:\R^3\mapsto \R^3$ is helical, if
\be
\label{eq:vesc}
\mathbf{v}(S_\r\x)=R_\r \mathbf{v}(\x), \quad \forall\r\in \R.
\ee
\end{defin}
\begin{claim}
\label{cl:helvec}
A continuously differentiable vector field $\mathbf{v}=(v_x,v_y,v_z)^T:\R^3\rightarrow \R^3$ is helical, if and only if it obeys the following relations:
\bse
\label{eq:helvec2}
\be
\der{v_x}{\vec{\t}}=v_y,
\ee
\be
\der{v_y}{\vec{\t}}=-v_x,
\ee
\be
\der{v_z}{\vec{\t}}=0.
\ee
\ese
\end{claim}
\begin{proof}
Assume the relation (\ref{eq:vesc}) and differentiate it with respect to $\r$ at $\r=0$.
In the right-hand side, only $R_\r$ depends on $\r$ and we have:
\be
\diff{R_\r}{\r}\Bigl|_{\r=0}=\begin{pmatrix} 0& 1 &0 \\-1 & 0 & 0 \\ 0 & 0 &0  \end{pmatrix}.
\ee
In the left-hand side, by differentiating we obtain a componentwise differentiation along $\vec{\t}$. Comparing the two expression, we get the required equality. To prove the converse, we apply the existence and uniqueness theorem for a system of ordinary differential equations in a similar manner to the proof in the case of the scalar helical function in Claim \ref{cl:charhelf}.
\end{proof}
We will introduce the following notation
\be
\label{eq:rmat}
\Rm=\diff{R_\r}{\r}\Bigl|_{\r=0}=\begin{pmatrix} 0& 1 &0 \\-1 & 0 & 0 \\ 0 & 0 &0  \end{pmatrix}.
\ee
With this notation, the previous claim can be stated in the following manner: a continuously differentiable vector field $\mathbf{v}$ is helical, if and only if $\mathbf{v}$ obeys the following relation
\be
\label{eq:helvchar}
\der{\mathbf{v}}{\vec{\t}}= \Rm\mathbf{v}.
\ee
\par
Let us write the relation that helical vector fields obey, i.e. (\ref{eq:helvec2}) in a more explicit manner:
\bse
\label{eq:vhel}
\begin{align}
\label{eq:uxhel}
y\der{v_x}{x}-x\der{v_x}{y}+\k\der{v_x}{z}&=v_y.\\
\label{eq:uyhel}
y\der{v_y}{x}-x\der{v_y}{y}+\k\der{v_y}{z}&=-v_x.\\
\label{eq:uzhel}
y\der{v_z}{x}-x\der{v_z}{y}+\k\der{v_z}{z}&=0.
\end{align}
\ese
\begin{defin} Let $\u,p$ be a solution of Euler equations then we call $\u,p$ a helical solution, if $\u$ is a helical vector field and $p$ is a helical scalar function.
\end{defin}
\subsection{The orthogonality condition}
We will introduce the following notation, for a vector field $\ve=(v_x,v_y,v_z)^T$, we will define a function $v_{\vec{\t}}$ by:
\be
v_{\vec{\t}}=(\mathbf{v},\vec{\t})=yv_x-xv_y+\k v_z.
\ee
\begin{lemma} Let $\u,p$ be a smooth helical solution of the Euler equations (\ref{eq:Ee}), then:
\label{lm:orthodyn}
\be
\frac{Du_{\vec{\t}}}{Dt}=F_{\vec{\t}},
\ee
where $\frac{D}{Dt}$ is the material time derivative, defined by $\frac{Df}{Dt}=\der{f}{t}+u_x\der{f}{x}+u_y\der{f}{y}+u_z\der{f}{z}$.
\end{lemma}
\begin{proof}
The full form of the three-dimensional Euler equations is:
\bse
\label{eq:E}
\be
\label{eq:Ex}
\der{u_x}{t}+u_x\der{u_x}{x}+u_y\der{u_x}{y}+u_z\der{u_x}{z}=-\der{p}{x}+F_x,
\ee
\be
\label{eq:Ey}
\der{u_y}{t}+u_x\der{u_y}{x}+u_y\der{u_y}{y}+u_z\der{u_y}{z}=-\der{p}{y}+F_y,
\ee
\be
\label{eq:Ez}
\der{u_z}{t}+u_x\der{u_z}{x}+u_y\der{u_z}{y}+u_z\der{u_z}{z}=-\der{p}{z}+F_z.
\ee
\ese
We take $x$ times equation (\ref{eq:Ex}) $-y$ times equation (\ref{eq:Ey}) and $\k$ times equation (\ref{eq:Ez}). The right-hand side of ({\ref{eq:E}) gives
\be
-\der{p}{\vec{\t}}+F_{\vec{\t}}=F_{\vec{\t}},
\ee
where we used the fact that $p$ is a helical function.
From the left-hand side of (\ref{eq:E}) we get:
\be
\begin{split}
\der{u_{\vec{\t}}}{t}&+y\left[u_x\der{u_x}{x}+u_y\der{u_x}{y}+u_z\der{u_x}{z}\right]-x\left[u_x\der{u_y}{x}+u_y\der{u_y}{y}+u_z\der{u_y}{z}\right]\\
&+\k\left[u_x\der{u_z}{x}+u_y\der{u_z}{y}+u_z\der{u_z}{z}\right]
\\
&=\der{u_{\vec{\t}}}{t}+u_x\der{u_{\vec{\t}}}{x}+u_xu_y+u_y\der{u_{\vec{\t}}}{y}-u_xu_y+u_z\der{u_{\vec{\t}}}{z}=\frac{Du_{\vec{\t}}}{Dt}.
\end{split}
\ee
By equating the two sides we conclude the lemma.
}
\end{proof}
\begin{corr}
\label{cor:uksi}
Suppose $\u_0$, $\mathbf{F}$ are helical vector fields that give rise to a helical smooth solution $\u,p$ of 3D Euler equations. If $F_{\vec{\t}}=0$ and $u_{0,{\vec{\t}}}=0$ then $u_{\vec{\t}}(t)=0,\forall t$.
\end{corr}
\begin{defin} We will say that a vector field $\mathbf{v}$ is orthogonal to the helices, if
\label{df:orthoh}
\be
\label{eq:ortho}
v_{\vec{\t}}=yv_x-xv_y+\k v_z=0.
\ee
\end{defin}
\begin{defin}[Orthogonality Assumption]
\label{df:ortho}
Throughout this article we will assume that the forcing field $\mathbf{F}$  and the velocity field of the solution of the Euler equations $\u$ are orthogonal to the helices in the sense of Definition \ref{df:orthoh}. Namely,
\be
F_{\vec{\t}}=0,
\ee
\be
u_{\vec{\t}}=yu_x-xu_y+\k u_z=0,
\ee
for interval of time for which the solution exists.
\end{defin}
\begin{rema}
By Lemma \ref{lm:orthodyn}, it is enough to assume that $F_{\vec{\t}}=0$ for all times and $u_{0,\vec{\t}}=(\vec{\t},\u_0)=0$ in order to satisfy the previous definition.
\end{rema}
The following two lemmas are the direct consequences of the Orthogonality Assumption (Definition \ref{df:ortho}) about the velocity field on it's vorticity $\mathbf{\O}=\nabla\wedge\u$.
\begin{lemma}
\label{lm:repo}
Let $\u$ be a $C^2$ helical vector field, such that $u_{\vec{\t}}=0$. Denote $\mathbf{\O}=\nabla\wedge\u=(\O_x,\O_y,\O_z)^T$, the vorticity of $\u$ then
\be
\label{eq:repo}
\mathbf{\O}=\frac{\o}{\k}\vec{\t}=(y\o,-x\o,\k\o)^T,
\ee
for a helical scalar function $\o=\O_z=\der{u_y}{x}-\der{u_x}{y}$.
\end{lemma}
\begin{proof}
Observe that if representation (\ref{eq:repo}) holds true then since $\mathbf{\O}$ has zero divergence, we have
\be
y\der{\o}{x}-x\der{\o}{y}+\k\der{\o}{z}=\der{\O_x}{x}+\der{\O_y}{y}+\der{\O_z}{z}=\nabla\cdot\OO=0.
\ee
Therefore by Claim \ref{cl:charhelf}, $\o$ will be a helical function. We compute the components of $\mathbf{\O}$:
Observe that equation (\ref{eq:ortho}) can be rewritten as:
\be
\label{eq:eluz}
u_z=\frac{1}{\k}(-yu_x+xu_y).
\ee
Also, by using equations (\ref{eq:uxhel}) and (\ref{eq:uyhel}), we can conclude the following equalities:
\be
\label{eq:eluxz}
\der{u_x}{z}=\frac{1}{\k}(-y\der{u_x}{x}+x\der{u_x}{y}-u_y),
\ee
\be
\label{eq:eluyz}
\der{u_y}{z}=\frac{1}{\k}(-y\der{u_y}{x}+x\der{u_y}{y}+u_x).
\ee
Therefore, we use (\ref{eq:eluz}) and (\ref{eq:eluxz}) to obtain:
\be
\begin{split}
\O_x&=\der{u_z}{y}-\der{u_y}{z}=\frac{1}{\k}\der{}{y}\left[-yu_x+xu_y\right]-\frac{1}{\k}\left[u_x+y\der{u_y}{x}+x\der{u_y}{y}\right]\\
&=\frac{1}{\k}\left[-u_x-y\der{u_x}{y}+x\der{u_y}{y}+u_x+y\der{u_y}{x}-x\der{u_y}{y}\right]=\frac{y}{\k}\left[\der{u_y}{x}-\der{u_x}{y}\right],
\end{split}
\ee
and we use (\ref{eq:eluz}) and (\ref{eq:eluyz}) to obtain:
\be
\begin{split}
\O_y&=\der{u_x}{z}-\der{u_z}{x}=\frac{1}{\k}\left[u_y-y\der{u_x}{x}+x\der{u_x}{y}\right]-\frac{1}{\k}\der{}{x}\left[-yu_x+xu_y\right]\\
&=\frac{1}{\k}\left[u_y-y\der{u_x}{x}+x\der{u_x}{y}+y\der{u_x}{x}-u_y-x\der{u_y}{x} \right]
=\frac{-x}{\k}\left[\der{u_y}{x}-\der{u_x}{y}\right],
\end{split}
\ee
\be
\O_z=\der{u_y}{x}-\der{u_x}{y}.
\ee
We see from the calculations above that
\be
\mathbf{\O}=\vec{\t}\frac{1}{\k}\left[\der{u_y}{x}-\der{u_x}{y}\right],
\ee
proving that $\mathbf{\O}$ has the representation (\ref{eq:repo}) with $\o=\der{u_y}{x}-\der{u_x}{y}=\O_z$.
\end{proof}
\begin{rema}
Observe that the previous lemma also proves that if $\F$ is a body forcing term, which obeys orthogonality condition, then the vorticity forcing term $\nabla\wedge\F$ obeys
\be
    \nabla\wedge\F=\frac{\vec{\t}}{\k}(\nabla\wedge\F)_z.
\ee
\end{rema}
\begin{lemma}
The orthogonality condition has also direct implications to the dynamics of the vorticity.
\label{lm:odyn}
Let $\OO=\nabla\wedge\u$ be the vorticity field of a smooth helical solution of the Euler equations which obey the Orthogonality Assumption (Definition \ref{df:ortho}), then $\OO$ obeys the following equation
\be
\frac{D\OO}{Dt}+\frac{1}{\k}\O_z\Rm\u=\nabla\wedge\mathbf{F}.
\ee
The matrix $\Rm$ is defined in equation (\ref{eq:rmat}).
\end{lemma}
\begin{proof}
From (\ref{eq:Evort}), the vorticity $\O$ satisfies
\be
\der{\OO}{t}+(\u\cdot\nabla)\OO+(\OO\cdot\nabla)\u=\nabla\wedge\mathbf{F}.
\ee
By Lemma \ref{lm:repo} , we have
\begin{align}
\O_x&=\frac{y}{\k} \O_z.\\
\O_y&=-\frac{x}{\k} \O_z.
\end{align}
Therefore,
\be
\begin{split}
(\OO\cdot\nabla)\u&=\O_x\der{\u}{x}+\O_y\der{\u}{y}+\O_z\der{\u}{z}
\\
&=\frac{1}{\k}\left(y\O_z\der{\u}{x}-x\O_z\der{\u}{y}+\k\O_z\der{\u}{z}\right)=\frac{1}{\k}\O_z\der{\u}{\vec{\t}}=\frac{1}{\k}\O_z\Rm\u,
\end{split}
\ee
where we used the fact that $\u$ is a helical vector field and Claim \ref{cl:helvec}.
\end{proof}
\begin{corr}
\label{cr:transo}
The helical scalar function $\o=\O_z$, where $\OO=\nabla\wedge\u$, the vorticity field of a smooth solution of the Euler equations, obeys the equation
\be
\label{eq:transo}
\frac{D\o}{Dt}=\der{F_y}{x}-\der{F_x}{y}.
\ee
\end{corr}
\begin{rema}
Observe from (\ref{eq:repo}) that
\be
\o={\O_z}=|\mathbf{\O}|\frac{\k}{\sqrt{\k^2+x^2+y^2}}.
\ee
Therefore, the equivalent of the previous two lemmas in the axi-symmetric setting would be the fact that if the azimuthal component of the velocity is zero -  $u_{\theta}=0$ then $\mathbf{\O}=\o_\theta \hat{\theta}$ and $\frac{D}{Dt}\left(\frac{\o_\theta}{r}\right)=(\nabla\wedge\mathbf{F})_{\theta}$.
\end{rema}
\subsection{Reduction to the two-dimensional equations}
\label{sbsc:red}
The purpose of this subsection is to setup the reduction of the three-dimensional problem to the problem with two components of velocity fields on a two dimensional domain. To achieve this goal,
we rewrite the orthogonality condition (equation (\ref{eq:ortho})) in the following form:
\be
\label{eq:elvz}
u_z=\frac{1}{\k}(-yu_x+xu_y).
\ee
In the equations (\ref{eq:vhel}), which characterize the helical vector field, we move differentiation by $z$ to the left-hand side:
\bse
\label{eq:eldz}
\begin{align}
\k\der{u_x}{z}&=u_y-y\der{u_x}{x}+x\der{u_x}{y},\\
\k\der{u_y}{z}&=-u_x-y\der{u_y}{x}+x\der{u_y}{y},\\
\k\der{u_z}{z}&=-y\der{u_z}{x}+x\der{u_z}{y}.
\end{align}
\ese
From equations (\ref{eq:elvz}) and (\ref{eq:eldz}) it is clear that we can eliminate the $u_z$ component from the velocity field and also all the differentiation by $z$. It means that we can rewrite our problem, as a problem for a two-dimensional vector field $\u_{\text{red}}=(u_x,u_y)$, with the fluid dynamics occuring in a two-dimensional domain of a form $\D_{\r_0}=\D\cap\{z=\r_0\}$. We will make an arbitrary choice of $\r_0=0$. We define a subset of $\R^2$ by
\be
\D_0=\{(x,y)^T|(x,y,0)^T\in \D \}.
\ee
We will solve the new equations on $\D_0$ and then we can recover the full vector field on $\D$ by the following algorithm:
\begin{description}
\item[Input]
\be
u_x(x,y,0;t),u_y(x,y,0;t)\quad \text{on }\D_0\times\{0\}.
\ee
\item[Step One] Recover the $u_z$ component on $\D_0\times\{0\}$ by:
\be
u_z(x,y,0;t)=\frac{1}{\k}(-yu_x(x,y,0;t)+xu_y(x,y,0;t)).
\ee
\item[Step Two] Recover the velocity field on $\D$ from the velocity field on $\D_0$ by:
\be
\label{eq:recvec}
\u(x,y,z;t)=R_{\frac{z}{\k}}\u(S_{-\frac{z}{\k}}\x;t),
\ee
where $\x=(x,y,z)^T$.
\end{description}
Observe that in a similar manner, we eliminate differentiation by $z$ from any helical function $f$ by using equation (\ref{eq:hesc}):
\be
\k\der{f}{z}=-y\der{f}{x}+x\der{f}{y}.
\ee
\begin{defin}
\label{df:fnext}
For a function $\phi:\D_0\rightarrow \R$ denote by $\tilde{\phi}$ the helical function, which extends $\phi$ to $\D$:
\be
\tilde{\phi}(x,y,z)=\phi(S_{-\frac{z}{\k}}(x,y,z)^T).
\ee
For a helical function $\psi:\D\rightarrow \R$ denote it's contraction to $\D_0$ by $\bar{\psi}$.
\be
\bar{\psi}(x,y)=\psi(x,y,0).
\ee
\end{defin}

\subsection{Effective vorticity-stream function formulation for helical solutions of the Euler Equations}
We will now introduce the effective stream function formulation on the domain $\D_0$. We rewrite the incompressibility condition:
\be
\begin{split}
\der{u_x}{x}&+\der{u_y}{y}+\der{u_z}{z}=\der{u_x}{x}+\der{u_y}{y}-\frac{y}{\k}\der{u_z}{x}+\frac{x}{\k}\der{u_z}{y}\\&=\der{u_x}{x}+\der{u_y}{y}-\frac{y}{\k^2}\der{}{x}(-yu_x+xu_y)+\frac{x}{\k^2}\der{}{y}(-yu_x+xu_y)
\\
&=\frac{1}{\k^2}\der{}{x}[(\k^2+y^2)u_x-xyu_y]+\frac{1}{\k^2}\der{}{y}[(\k^2+x^2)u_y-xyu_x].
\end{split}
\ee
We used equation (\ref{eq:uzhel}) in the first equality and then equation (\ref{eq:elvz}) in the second equality. Therefore the incompressibility for a reduced velocity field is:
\be
\frac{1}{\k^2}\der{}{x}[(\k^2+y^2)u_x-xyu_y]+\frac{1}{\k^2}\der{}{y}[(\k^2+x^2)u_y-xyu_x]=0
\ee
We introduce the function $\psi:\D_0\rightarrow \R$ by
\bse
\label{eq:dfpsi}
\be
\der{\psi}{y}=-\frac{1}{\k^2}\left[(\k^2+y^2)u_x-xyu_y\right],
\ee
\be
\der{\psi}{x}=\frac{1}{\k^2}\left[-xyu_x+(\k^2+x^2)u_y\right].
\ee
\ese
Let us rewrite the relation between $\psi$ and $u_x,u_y$:
\be
\begin{pmatrix}\der{\psi}{x}\\ \der{\psi}{y}\end{pmatrix}=\frac{1}{\k^2}\begin{pmatrix} -xy & \k^2+x^2 \\ -\k^2-y^2 & xy  \end{pmatrix}\begin{pmatrix}u_x \\ u_y\end{pmatrix}.
\ee
Now, we can invert the previous equation to obtain:
\be
\label{eq:upsi}
\begin{pmatrix}u_x \\ u_y\end{pmatrix}=\frac{1}{\k^2+x^2+y^2}\begin{pmatrix} xy & -\k^2-x^2 \\ \k^2+y^2 & -xy  \end{pmatrix}\begin{pmatrix}\der{\psi}{x}\\ \der{\psi}{y}\end{pmatrix}.
\ee
\begin{rema}
In fact, the function $\psi$ has an intrinsic meaning to the full three-dimensional problem. Consider the three-dimensional stream function of the problem, defined by $\nabla\wedge\PS=\u$. Unlike the planar and axi-symmetric problem, it is no longer true that $\PS=\vec{\t}\tilde{\psi}$, where $\tilde{\psi}$ is the helical extension of $\psi$ in the sense of Definition \ref{df:fnext}. But $\tilde{\psi}$ can be recovered as a component of $\PS$ in a certain non-orthogonal decomposition. We leave this point out, since we concentrate on the analytical aspects of our problem.
\end{rema}
We wish to establish the boundary conditions for $\psi$ on $\partial\D_0$. We will employ the boundary condition for $\u$, namely that $(\u(\x), \mathbf{n}(\x))=0$ for $\x\in\partial \D$, where $\mathbf{n}(x)$ is a normal to $\partial\D$ at $\x$. Observe that while in the two-dimensional case, the fact that  the velocity field $\u$ is orthogonal to the normal to the boundary automatically implies that it is proportional to the tangent to the boundary, it is no longer true in the three-dimensional case. But we will exploit the helical symmetry and the orthogonality condition to compute the tangent vector to $\partial \D_0$.
\begin{lemma}
\label{lm:tng}
Let $\D$ be a helical simply connected domain with twice differentiable boundary and let $\D_0=\{(x,y)^T|(x,y,0)^T\in\D\}$. Let $(x,y)^T\in\partial\D_0$ and $\mathbf{t}=(t_x,t_y)^T$ be a tangent vector to $\partial \D_0$ (with respect to $\R^2$) then $\mathbf{t}$ at the point $(x,y)^T$ is proportional to
\[
\mathbf{v}=(u_x(\k^2+y^2)-xyu_y,u_y(\k^2+x^2)-xyu_x)^T,
\]
provided that the vector $\mathbf{v}$ is not zero.
\end{lemma}
\begin{corr}[Boundary conditions for $\psi$]
\label{cr:psibc}
Let $\D_0\in \R^2$ be a simply connected domain with twice-differentiable boundary, then one can choose $\psi$ such that
\be
\psi|_{\partial\D_0}=0.
\ee
\end{corr}
\begin{proof}[Proof of Corrolary \ref{cr:psibc}]
By Lemma \ref{lm:tng}, we have on the boundary of $\D_0$
\be
\begin{split}
\der{\psi}{\mathbf{t}}&=\der{\psi}{x}t_x+\der{\psi}{y}t_y
\\&=c(-[(\k^2+y^2)u_x-xyu_y][u_x(\k^2+y^2)-xyu_y]
\\&+[u_x(\k^2+y^2)-xyu_y][u_y(\k^2+x^2)-xyu_x])=0,
\end{split}
\ee
where we used Equations \ref{eq:dfpsi}, which defined $\psi$. The constant $c$ is the proportionality constant, which depends on the point $(x,y)^T$. If the conditions of the lemma do not hold, namely $(u_x(\k^2+y^2)-xyu_y,u_y(\k^2+x^2)-xyu_x)=0$ then $\nabla\psi(x,y)=0$. Therefore, $\psi$ is constant on the boundary of a simply connected domain at every point $(x,y)^T\in\partial\D_0$. And since $\psi$ is defined only up to a constant, we choose $\psi|_{\partial\D_0}=0$.
\end{proof}
\begin{proof}[Proof of Lemma \ref{lm:tng}]
Let $\mathbf{n}=(n_x,n_y,n_z)$ be a normal to $\partial\D$ at the point $(x,y,0)^T\in\partial\D_0\times{\{0\}}\subseteq \partial\D$. Since $\D$ is a helical domain, then all the points of the form $S_{\r}((x,y,0)^T)$ belong to $\partial\D$. Therefore, the field of tangents to the helices $\vec{\t}$ is tangent to the boundary of $\D$. By definition, $\mathbf{n}$ is normal to all the tangents to $\partial\D$. Thus we conclude that
\be
(\vec{\t},\mathbf{n})=0,
\ee
therefore
\be
 yn_x-xn_y+\k n_z=0.
\ee
This means that $\mathbf{n}$ has a form
\be
\label{eq:nform}
\mathbf{n}=(n_x,n_y, -\frac{y}{\k}n_x+\frac{x}{\k}n_y)^T.
\ee
\\
Also, by the orthogonality assumtpion, the velocity field $\u$ is also orthogonal to the helices and thus $\u$ also has the form
\be
\label{eq:uform}
\u=(u_x,u_y,-\frac{y}{\k}u_x+\frac{x}{\k}u_y)^T.
\ee
But, since $\u$ obey the no-flow boundary condition, then $(\u,\mathbf{n})=0$. We combine equations (\ref{eq:nform}) and (\ref{eq:uform}) with the boundary condition to get
\be
\begin{split}
0&=u_xn_x+u_yn_y+(-\frac{y}{\k}u_x+\frac{x}{\k}u_y)(-\frac{y}{\k}n_x+\frac{x}{\k}n_y)\\
&=\frac{1}{\k^2}\begin{pmatrix} u_x & u_y\end{pmatrix}
\begin{pmatrix}\k^2+y^2 &-xy\\ -xy  & \k^2+x^2 \end{pmatrix}
\begin{pmatrix} n_x \\ n_y\end{pmatrix}=\frac{1}{\k^2}\begin{pmatrix} u_x & u_y\end{pmatrix}G
\begin{pmatrix} n_x \\ n_y\end{pmatrix},
\end{split}
\ee
where $G$ is the matrix
\be
G=\begin{pmatrix}\k^2+y^2 &-xy\\ -xy  & \k^2+x^2 \end{pmatrix}.
\ee
Therefore, we conclude that the vector $\begin{pmatrix} n_x \\ n_y \end{pmatrix}$ is proportional to $G^{-1}\begin{pmatrix} -u_y \\ u_x \end{pmatrix}$. Computing:
\be
\begin{split}
G^{-1}\begin{pmatrix} -u_y \\ u_x \end{pmatrix}&=\begin{pmatrix}\k^2+y^2 &-xy\\ -xy  & \k^2+x^2 \end{pmatrix}^{-1}\begin{pmatrix} -u_y \\ u_x \end{pmatrix}\\&=\frac{1}{\k^2+x^2+y^2}\begin{pmatrix}\k^2+y^2 &-xy\\ -xy  & \k^2+x^2 \end{pmatrix}^{-1}\begin{pmatrix} -u_y \\ u_x \end{pmatrix}
\\&=\frac{1}{\k^2+x^2+y^2}\begin{pmatrix}-\k^2-x^2 &-xy\\ -xy  & -\k^2-y^2 \end{pmatrix}\begin{pmatrix} -u_y \\ u_x \end{pmatrix}\\
&=\frac{1}{\k^2+x^2+y^2}\begin{pmatrix}(\k^2+x^2)u_y-xyu_x\\xyu_y-(\k^2+y^2)u_x   \end{pmatrix}.
\end{split}
\ee
Observe that we stipulated that the vector $G^{-1}\begin{pmatrix} -u_y \\ u_x \end{pmatrix}$ is not zero in the conditions of the lemma. Therefore, the vector $(n_x,n_y)^T$ is proportional to the vector $((\k^2+x^2)u_y-xyu_x,xyu_y-(\k^2+y^2)u_x)^T$. Next, let $\mathbf{t}=(t_x,t_y)$ be vector tangent to $\partial\D_0$ at $(x,y)^T$ then the vector $(t_x,t_y,0)^T$ is tangent to $\D$ at the point $(x,y,0)^T$. This means that $(t_x,t_y,0)^T$ is orthogonal to the normal $\mathbf{n}$. From this we conclude
\be
((t_x,t_y,0)^T,\mathbf{n})=t_xn_x+t_yn_y=0.
\ee
Therefore $(t_x,t_y)$ is proportional to $(-n_y,n_x)$ which means that
\be
(t_x,t_y)=c ((\k^2+y^2)u_x-xyu_y,(\k^2+x^2)u_y-xyu_x),
\ee
for some constant $c$, which depends on $(x,y)^T$. This concludes the proof.
\end{proof}
We now wish to express the connection of the vorticity and the stream function and also the dynamics of the vorticity in the new formulation. By Lemma \ref{lm:repo}, we can concentrate on the dynamics of $\o=\OO_z$. By the same lemma, it is a helical function, which we can restrict to a function on $\D_0$, which we denote also by $\o$. Using equation (\ref{eq:upsi}), we represent $\o$ in terms of $\psi$:
\be
\begin{split}
\o&=\der{u_x}{y}-\der{u_y}{x}\\
&=\begin{pmatrix} \partial_y& -\partial_x \end{pmatrix}\frac{1}{\k^2+x^2+y^2}\begin{pmatrix} -xy & \k^2+x^2 \\ -\k^2-y^2 & xy  \end{pmatrix}\begin{pmatrix}\der{\psi}{x}\\ \der{\psi}{y}\end{pmatrix}\\
&=\begin{pmatrix} \partial_x& \partial_y \end{pmatrix}\begin{pmatrix} 0 & -1\\ 1 & 0\end{pmatrix}\frac{1}{\k^2+x^2+y^2}\begin{pmatrix} -xy & \k^2+x^2 \\ -\k^2-y^2 & xy  \end{pmatrix}\begin{pmatrix}\der{\psi}{x}\\ \der{\psi}{y}\end{pmatrix}\\
&=\begin{pmatrix} \partial_x& \partial_y \end{pmatrix}\frac{1}{\k^2+x^2+y^2}\begin{pmatrix} \k^2+y^2 &- xy   \\  -xy & \k^2+x^2 \end{pmatrix}\begin{pmatrix}\der{\psi}{x}\\ \der{\psi}{y}\end{pmatrix}\\
&=\begin{pmatrix} \partial_x& \partial_y \end{pmatrix}\frac{1}{\k^2+x^2+y^2}\begin{pmatrix} \k^2+y^2 &- xy   \\  -xy & \k^2+x^2 \end{pmatrix}\begin{pmatrix}\partial_x\\ \partial_y\end{pmatrix} \psi.
\end{split}
\ee
Let us introduce a notation for the matrix-valued function, which appears in the previous equation
\be
\label{eq:K}
K(x,y)=\frac{1}{\k^2+x^2+y^2}\begin{pmatrix} \k^2+y^2 &- xy   \\  -xy & \k^2+x^2 \end{pmatrix}.
\ee
We will denote the connection between $\o$ and $\psi$ by an operator $\LH$.
\be
\o=\LH\psi.
\ee
With the notation for $K$, the operator $\LH$ can be written:
\be
\LH\psi=\nabla\cdot (K\nabla\psi)=\text{div}(K\text{grad}\psi).
\ee
In the last step, we will write the equation for $\o$ using the stream function formulation.
\begin{lemma}[Dynamics of the vorticity in the stream function formulation]
The vorticity function $\o$ obeys the following equation:
\be
\label{eq:dynohel}
\der{\o}{t}-\der{\psi}{y}\der{\o}{x}+\der{\psi}{x}\der{\o}{y}=f,
\ee
where $f=\der{F_y}{x}-\der{F_x}{y}$.
\end{lemma}
\begin{proof}
Our starting point is Corrolary \ref{cr:transo}:
\be
\label{eq:vortdyn}
\frac{D\o}{Dt}=f,
\ee
We have:
\be
\begin{split}
\frac{D\o}{Dt}&=\der{\o}{t}+u_x\der{\o}{x}+u_y\der{\o}{y}+u_z\der{\o}{z}\\
&=\der{\o}{t}+u_x\der{\o}{x}+u_y\der{\o}{y}+\frac{1}{\k^2}(-yu_x+xu_y)(-y\der{\o}{x}+x\der{\o}{y})\\
&=\der{\o}{t}+\frac{1}{\k^2}\begin{pmatrix}u_x& u_y \end{pmatrix}\begin{pmatrix}\k^2+y^2& -xy\\ -xy& \k^2+x^2 \end{pmatrix}\begin{pmatrix}\der{\o}{x}\\ \der{\o}{y} \end{pmatrix}.
\end{split}
\ee
We now apply the stream function formulation (\ref{eq:upsi}):
\begin{multline}
\frac{D\o}{Dt}
=\der{\o}{t}+\\
+\frac{1}{\k^2(\k^2+x^2+y^2)}\begin{pmatrix}\der{\psi}{x}& \der{\psi}{y} \end{pmatrix}\begin{pmatrix} xy & \k^2+y^2\\ -\k^2-x^2  & -xy  \end{pmatrix}\begin{pmatrix}\k^2+y^2& -xy\\ -xy& \k^2+x^2 \end{pmatrix}\begin{pmatrix}\der{\o}{x}\\ \der{\o}{y} \end{pmatrix}=\\
=\der{\o}{t}+\begin{pmatrix}\der{\psi}{x}& \der{\psi}{y} \end{pmatrix}\begin{pmatrix}0 & 1\\ -1& 0 \end{pmatrix}\begin{pmatrix}\der{\o}{x}\\ \der{\o}{y} \end{pmatrix}=\der{\o}{t}-\der{\psi}{y}\der{\o}{x}+\der{\psi}{x}\der{\o}{y}.
\end{multline}
Inserting the last equality into equation (\ref{eq:vortdyn}) proves the Lemma.
\end{proof}
\section{Functional Setting and Weak formulation}
\label{sc:fun}
\subsection{Function Spaces}
In this subsection, we will consider measurable functions and the appropriate functional space.
We will understand helical measurable functions and helical vector fields as in Definitions \ref{df:helf} and \ref{df:helv} with the equality being true for almost every $\x$.
\par
We will again use the notation $\D_0=\{(x,y)|(x,y,0)\in \D\}$. Observe that the assumption on $\D$ imply the following for $\D_0$:
\begin{itemize}
\item $\D_0$ is simply connected,
\item $\D_0$ is bounded,
\item $\D_0$ has a twice differentiable boundary.
\end{itemize}
For a helical integrable function $f:\D\rightarrow \R$, the appropriate reduction is $\bar{f}:\D_0\rightarrow \R$, defined by
\be
\bar{f}(x,y)=\frac{1}{2\pi}\int\limits_{0}^{2\pi}{f(S_{\r}(x,y,0))d\r},
\ee
clearly for $g$ continuous, we will have $\bar{g}=g|_{\D_0\times{0}}$. The normalization constant comes from the computation with $f\equiv 1$.
In what follows, we will identify $f$ with $\bar{f}$ and think of $f$ as being a function on $\D_0$. It can be proven that the spaces of functions, which we define for function on $\D_0$ are equivalent to the spaces of helical functions on $\D$ but we will not need this point in our analysis. The only fact, which is obvious in our reduction is
\be
\text{ess}\sup\limits_{\D}f=\text{ess}\sup\limits_{\D_0}\bar{f},
\ee
which says that the class of essentialy bounded helical functions on $\D$ is equivalent to the class of essentialy bounded on $\D_0$.
\begin{defin} Denote by $L^p(\D_0)$ the space of (real-valued) measurable functions on $\D_0$ with the norm
\be
\|f\|^p_p=\int\limits_{\D_0}{|f(x,y)|^pdxdy}.
\ee
\end{defin}
For the inner product in $L^2(\D_0)$ we will use the notation:
\be
\la f,g \ra=\int\limits_{\D_0}{f(x,y)g(x,y)dxdy}.
\ee
\begin{defin} Denote by $W^{n,p}(\D_0)$ the space of measurable functions on $\D_0$ whose distributional derivatives up to order $n$ are $L^p(\D_0)$ functions with the norm
\be
\|f\|^p_{W^{n,p}(\D_0)}=\sum\limits_{|\alpha|\leq n} \left\|\der[\alpha]{f}{x}\right\|^p_p.
\ee
Also, the space $H^1_0(\D_0)$ is the space of $W^{1,2}(\D_0)$ function, which vanish at the boundary, with the norm
\be
\|f\|_{H_0^1}=\int\limits_{\D_0}{\left|\der{f}{x} \right|^2+\left|\der{f}{y} \right|^2dxdy}.
\ee
\end{defin}
We will also introduce the following twisted inner product
\begin{defin}["Helical" inner product]

\be
\begin{split}
\lh f,g \rh&=\int\limits_{\D_0}{\frac{1}{\k^2+x^2+y^2}\begin{pmatrix} \der{f}{x} & \der{f}{y} \end{pmatrix}\begin{pmatrix} \k^2+y^2 & -xy\\ -xy& \k^2 + x^2 \end{pmatrix}\begin{pmatrix} \der{g}{x} \\ \der{g}{y} \end{pmatrix} dxdy}\\
&=\int\limits_{\D_0}{(\nabla f(x,y), K(x,y)\nabla g(x,y)) dxdy},
\end{split}
\ee
where function $K$ was defined in Equation \ref{eq:K}.
The appropriate norm is denoted by $\| f\|_h=\lh f, f \rh^{1/2}$.
\end{defin}
\begin{claim}
\label{cl:heleqv}
The norm $\|\cdot \|_h$ is equivalent to the following semi-norm:
\be
\|f\|_{H_0^1}^2=\int_{\D_0}{\left(\left|\der{f}{x}\right|^2+\left|\der{f}{y}\right|^2\right)dxdy}=\int_{\D_0}{(\nabla f,\nabla f)dxdy}.
\ee
\end{claim}
\begin{proof} The matrix valued function $K$ is symmetric and has eigenvalues $1$ and $\frac{\k^2}{\k^2+x^2+y^2}$. Therefore, since $\D_0$ is bounded, we have
\be
\frac{\k^2}{\k^2+R^2}(\nabla f,\nabla f)\leq (K\nabla f,\nabla f) \leq(\nabla f,\nabla f),
\ee
where $R$ is the maximal distance from 0 in $\D_0$. Integrating over $\D_0$ and taking square root proves the claim.
\end{proof}
\subsection{Elliptic Regularity}
\par In this subsection, we will quote the elliptic regularity theory, which was developed by Yudovich in a sequence of papers \cite{YEL1},\cite{YEL2}, based on the Calderon-Zygmund theory for singular integral operators \cite{St} and Agmon et al. \cite{ADN} results on the regularity of elliptic operators up to the boundary. Yudovich \cite{Yud} applied his elliptic regularity theory to the two-dimensional Euler equations. This theory is the main ingredient of our work. We quote his theorem \cite[Theorem 2.1]{Yud}
\par In what follows we will sometimes use $x_1,x_2$ for the variables $x,y$.
\begin{theorem}[Yudovich's Elliptic Regularity]
\label{th:orYEL}
Let $L$ be an elliptic operator with H\"older continuous coefficients. Let $D$ be a compact domain in $\R^n$ with twice differentiable boundary such that $L$ satisfies the following inequality
\be
\label{eq:PTineq}
\int\limits_{D}{u(Lu) dx}\geq m\|u\|^2_{H_0^1({D})}, \quad \forall u\in{H_0^1({D})},\quad m>0.
\ee
Then given $f\in L^p(D)$, there exists $q\in W^{2,p}(D)$ which solves the problem:
\be
Lq=f,\quad q|_{\partial D}=0.
\ee
The function $q$ obeys the estimate:
\be
\|q\|_{W^{2,p}(D)}\leq Cp\|f\|_{L^p(D)}.
\ee
\end{theorem}
\begin{defin}[Helical operator $\LH$]
\label{df:lh}
The operator $\LH$ is  defined by:
\be
\LH \psi= \sum\limits_{i,j=1}^2{\der{}{x_i}(K_{ij}\der{\psi}{x_j})},
\ee
where
\be
K(x_1,x_2)=\frac{1}{\k^2+x_1^2+x_2^2}\begin{pmatrix} \k^2+x_2^2 &- x_1x_2   \\  -x_1x_2 & \k^2+x_1^2 \end{pmatrix}.
\ee
\end{defin}
\begin{claim}
The operator $\LH$ obeys the following identity, whenever both sides are defined
\be
\label{eq:intp}
\lh f,g\rh=-\la \LH f, g \ra,
\ee
with $f,g$ vanishing at the boundary.
\end{claim}
\begin{corr}
\label{cr:YEL} Let $\LH$ be an operator from Definition \ref{df:lh} and $f\in L^p(\D_0)$ then there exists $\psi \in W^{2,p}(\D_0)$ that solves the following problem:
\be
\LH \psi=f\quad \psi|_{\partial \D_0}=0,
\ee
and there exists a constant $C$ independent of $p$ such that
\be
\|\psi\|_{W^{2,p}}\leq Cp\|f\|_{p}.
\ee
\end{corr}
\begin{proof}
Clearly, $\LH$ has bounded, smooth coefficients. Combine identity (\ref{eq:intp}) with Claim \ref{cl:heleqv} to verify the inequality in equation (\ref{eq:PTineq}).
\end{proof}
We will also need the following statement: (See \cite[Theorem 2.2]{Yud})
\begin{theorem}
\label{th:negEE}
Let $q$ be a solution of
\be
\LH q=\der{f}{x_k},\quad q|_{\partial \D_0}=0.
\ee
Then $q$ obeys the following estimate
\be
||q||_{W^{1,p}}\leq Cp||f||_{L_p}.
\ee
\end{theorem}
\subsection{The weak formulation and statement of the main theorem}
We will now formulate the weak problem. As a basis, we use Equation \ref{eq:dynohel}. Written in terms of the stream function $\psi$ and operator $\LH$ the equation reads
\be
\der{}{t}(\LH\psi)-\der{\psi}{y}\der{}{x}(\LH\psi)+\der{\psi}{x}\der{}{y}(\LH\psi)=f.
\ee
We define the weak problem in the following way.
\begin{defin}
\label{df:main}
Given $F:\D_0\times \inte\rightarrow\R^2 $ and $\psi_0:\D_0\rightarrow R$ integrable functions. Then  $\psi:\D_0\times \inte\rightarrow \R$ with integrable $\der{\psi}{x_i}$ and $\LH\psi$ is a solution of the weak Euler equation if it obeys the identity
\be
\begin{split}
\label{eq:weak}
\int\limits _{\D_0} {\LH\psi_0(x) \phi(x,0) dxdy}&-\int\limits_{0}^T{\int\limits_{\D_0}{\LH\psi(x,t)\der{\phi}{t}(x,t)dxdy}dt}+\int\limits_{0}^T{\int\limits_{\D_0}{\der{\psi}{y}\LH \psi \der{\phi}{x}dxdy}dt}\\
&-\int\limits_{0}^T{\int\limits_{\D_0}{\der{\psi}{x}\LH \psi \der{\phi}{y}dxdy}dt}
=\int\limits_{0}^T{\int\limits_{\D_0}{\left(\der{F_y}{x}-\der{F_x}{y} \right)\phi dxdy}dt},
\end{split}
\ee
for every smooth function $\phi$ with compact support in $\D_0\times\inte$.
\end{defin}
We are now ready to state the main theorem.
\begin{theorem}[Main theorem]
\label{th:main}
Let $F:\D_0\times\inte \rightarrow \R^2$ be such that $F,\der{F_x}{y}-\der{F_y}{x}\in L^{\infty}(\D_0\times\inte)$. Let $\psi_0\in W^{2,p}$ such that $\LH\psi_0\in L^{\infty}(\D_0)$ and $\psi_0|_{\partial\D_0}=0$ then there exists a unique $\psi:\D_0\times\inte\rightarrow \R$ with $\LH\psi\in L^{\infty}(\D_0\times\inte)$ and $\psi|_{\partial\D_0}=0 ,\forall t\in\inte$, which solves the weak Euler equation in the sense of Definition \ref{df:main}.
\end{theorem}
\begin{rema} We will also prove that the problem is well posed in $H^1_0$ norm.
\end{rema}
\begin{rema} Using the ideas of Kato \cite{Kpr}, we will prove prove that the solution is unique within the class of all stream functions $\psi$ with $\LH\psi\in L^{\infty}(\inte,L^2(\D_0))$.
\end{rema}
We wil prove uniqueness in Section \ref{sc:uni} and existence in Section \ref{sc:exi}.
\section{Proof of Uniqueness}
\label{sc:uni}
\par In this section, we will prove the uniqueness part of the Theorem \ref{th:main}. We will employ an energy-principle type of argument, which is identical to that of Yudovich \cite[section 3]{Yud}. We follow closely his idea, retaining even some of his notations.
\par We will need two lemmas.
\begin{lemma}
\label{lm:tm}
A weak solution of the Euler Equation is a solution of the problem for any cylinder $\D_0\times[t_1,t_2]$, for $0\leq t_1<t_2\leq T$.
\end{lemma}
\begin{proof}
We substitute into the integral relation a test function of the form $h(t)\phi(x,t)$. We have the following equality
\be
-h(0)\int\limits_{\D_0}{(\LH\psi_0)(x,y)\phi(x,y,0)dxdy}+\int\limits_0^T{h'(t)M_t(\psi,\phi)dt}+\int\limits_{0}^T{h(t)N_t(\psi,\phi)dt}=0,
\ee
where
\be
\label{eq:defM}
M_t(\psi,\phi)=-\int\limits_{\D_0}{(\LH\psi)\phi dxdy},
\ee
\be
\label{eq:defN}
N_t(\psi,\phi)=-\int\limits_{\D_0}{\left(\LH\psi\right)\der{\phi}{t}+\LH\psi(\der{\psi}{x}\der{\phi}{y}-\der{\psi}{y}\der{\phi}{x}) dxdy}.
\ee
$M_t$ and $N_t$ are essentially bounded measurable functions. If we choose $h(0)=0$ then
\be
\diff{}{t}M_t=N_t,
\ee
in the sense of weak derivatives. Therefore $M_t$ is weakly differentiable and as such it is absolutely continuous. Thus we have
\be
M_{t_2}-M_{t_1}=\int\limits_{t_1}^{t_2}{N_\tau d\tau}.
\ee
Substituting the definition of $M,N$, equations (\ref{eq:defM}),(\ref{eq:defN}) into the last equation gives
\begin{multline}
\int\limits_{\D_0}{\LH\psi(x,y,t_1)\phi(x,y,t_1)dxdy}-\int\limits_{\D_0}{\LH\psi(x,y,t_2)\phi(x,y,t_2)dxdy}+\\
+\int\limits_{t_1}^{t_2}{\int\limits_{\D_0}{(\LH\psi)\der{\phi}{t}+(\LH\psi)(\der{\psi}{x}\der{\phi}{y}-\der{\psi}{y}\der{\phi}{x}) dxdy}}=0,
\end{multline}
which means that $\psi$ is the weak solution in the cylinder $\D_0\times [t_1,t_2]$.
\end{proof}
\begin{lemma}
\label{lm:psidermix}
Let  $\psi(x,y;t)$ be a weak solution. Then weak derivatives $\dermix{\psi}{x}{t}$,$\dermix{\psi}{y}{t}$ exist and for any $p\geq p_0>1$
\begin{multline}
 \max\limits_{0\leq t\leq T}\Bigl|\Bigl|\dermix{\psi}{x}{t}\Bigr|\Bigr|_{L^p(\D_0)},\max\limits_{0\leq t\leq T}\Bigl|\Bigl|\dermix{\psi}{y}{t}\Bigr|\Bigr|_{L^p(\D_0)}\leq \\
 \leq Cp \max\limits_{0\leq t\leq T}\left[||F||_{L^p(\D_0)}+||\LH\psi||_{L^{\infty}(\D_0)}||\psi||_{W^{1,p}(\D_0)}\right].
\end{multline}
\end{lemma}
\begin{proof}
$\der{\psi}{t}$ satisfies formally the following equations:
\be \LH\der{\psi}{t}=f-\der{}{x}\left(\LH\psi\der{\psi}{y}\right)+\der{}{y}\left(\LH\psi\der{\psi}{x}\right),\ee
\be
\der{\psi}{t}\Bigl|_{\partial\D_0}=0.
\ee
Take $r>2$ and build a sequence $\{g_k(x,y,t)\},k=1,2...$ such that $\LH g_k$ are continuously differentiable in $\D_0\times\inte, g_k|_{\partial \D_0}=0$ and $\LH g_k\longrightarrow \LH \psi$ in $L^{p,r}=L^r([0,T];L^p(\D_0))$ for every $p>0$. Define $q_k$ by:
\be
\label{eq:defap}
\LH q_k=\der{}{y}\left(\LH g_k\der{g_k}{x}\right)-\der{}{x}\left(\LH g_k\der{g_k}{y}\right)+\der{F_y}{x}-\der{F_x}{y}.
\ee
The application of Theorem (\ref{th:negEE}) gives the following bound:
\be
||q_k||_{W^{1,p}}\leq Cp\left(||F||_{L^p}+\Bigl|\Bigl|\LH g_k\der{g_k}{x}\Bigr|\Bigr|_{L^p}+\Bigl|\Bigl|\LH g_k\der{g_k}{y}\Bigr|\Bigr|_{L^p}\right).
\ee
Define
\be
\psi_k=\psi_0(x)+\int\limits_0^t{q_k(x,\tau)\,d\tau}.
\ee
Then
\be
\label{eq:estdert}
\Big\Arrowvert\dermix{\psi_k}{x}{t}\Big\Arrowvert_{L^{p,r}}^r\leq C_1^rp^r\left(\|F\|_{L^{p,r}}^r+\| \LH \psi\|_{L^{\infty}}^r\|\nabla \psi\|_{L^{p,r}}^r\right).
\ee
The same estimate holds for $\dermix{\psi_k}{y}{t}$.

By definition, the solutions of Equation \ref{eq:defap} obey:
\be
\begin{split}
-\lh q_k,\phi \rh&=\int\limits_{\D_0}{(\LH q_k)\phi dxdy}\\
&=\int\limits_{\D_0}{\left(-(\LH g_k) \der{g_k}{x}\right)\der{\phi}{y}
+\left(-(\LH g_k) \der{g_k}{y}\right)\der{\phi}{x}dxdy}
+\int\limits_{\D_0}{f\phi dxdy},
\end{split}
\ee
for any smooth test function $\phi(x,y,T)=\phi\big|_{\partial \D}=0$. We transform the relation:
\be -\lh q_k,\phi \rh = \lh \der{\psi_k}{t}, \phi \rh=-\diff{}{t} \lh \psi_k,\phi \rh - \la \LH\psi_k \der{\phi}{t}\ra.
\ee
We integrate this from $0$ to $T$ to get
\begin{multline}
\label{eq:approxweak}
\int\limits_{\D_0}{(K(x,y)\nabla \psi_{0},\nabla \phi(x,y,0) dxdy}+\\
+\int\limits_0^T{\int\limits_{\D_0}{-\LH\psi_k\der{\phi}{t}+\LH g_k\left(\der{g_k}{x}\der{\phi}{y}-\der{g_k}{y}\der{\phi}{x}\right)dxdydt}}=\int\limits_0^T{\la f,\phi \ra dt}.
\end{multline}
But according to the construction of $g_k$ we have
\begin{multline}
\label{eq:Bilim}
\lim\limits_{k\rightarrow \infty}{\int\limits_{\D_0\times\inte}{\LH g_k\left(\der{g_k}{x}\der{\phi}{y}-\der{g_k}{y}\der{\phi}{x}\right)dxdydt}}=\\
={\int\limits_{\D_0\times\inte}{\LH \psi\left(\der{\psi}{x}\der{\phi}{y}-\der{\psi}{y}\der{\phi}{x}\right)dxdydt}}.
\end{multline}
Therefore, by comparing Equation (\ref{eq:approxweak}) with the weak formulation - Equation (\ref{eq:weak}) and taking the limit, we conclude that:
\be
\lim\limits_{k\rightarrow \infty} \int\limits_0^T{\lh \psi_k,\der{\phi}{t}\rh}=\int\limits_0^T{\lh \psi,\der{\phi}{t}\rh}.
\ee
Since $\der{\phi}{t}$ are dense in $L^2([0,T],H^1_0)$ and the sequence $\psi_k$ is bounded in $L^2([0,T],H^1_0)$. Also $\psi_k$ are compact in $L^2(\D_0\times\inte)$ because $\der{\psi_k}{t}$ is bounded in $W^{1,p}$, by Equation \ref{eq:estdert} (see \cite[Lemma 8.4]{CF}). We conclude that $\psi_k$ converge weakly to $\psi$ in $L^2([0,T],H^1_0)$ and therefore they converge in norm in the space $L^2(\D_0\times \inte)$. Therefore, the derivatives of $\psi_k$ converge to  limits, which are weak derivatives of $\psi$. These derivatives obey the estimate
\be
\|\dermix{\psi}{x}{t}\|_{L^{p,r}}\leq Cp\left(\|F\|_{L^{p,r}}+\|\LH \psi\|_{L^{\infty}}\|\nabla \psi\|_{L^{p,r}}\right),
\ee
where we take the limit $r\rightarrow \infty$ to establish the required estimate.
\end{proof}
\begin{lemma}
Let $\psi_1$ be a solution of the Euler equation with $\LH\psi_1\in L^{\infty}(\D_0\times\inte)$ with initial condition $\psi_{01}$ with $\LH\psi_{01}\in L^{\infty}(\D_0)$. Then for every $\eta>0$ there exists $\delta>0$ such that if $\psi_2$ is a solution of the Euler equation with $\LH\psi_2\in L^{\infty}(\D_0\times\inte)$ with initial condition $\psi_{02}$ with $\LH\psi_{02}\in L^{\infty}(\D_0)$ then if $\|\psi_{01}-\psi_{02}\|_{H^1_0(\D_0)}<\delta$ then $\|\psi_1(T)-\psi_2(T)\|_{H^1_0(\D_0)}<\eta$. In particular, the solutions are unique.
\end{lemma}
\begin{proof}
Form
\be
\alpha=\psi_1-\psi_2.
\ee
We use Lemma \ref{lm:tm} for $0$ and $t$ and substract the weak formulations to get
\be
\begin{split}
\label{eq:ap}
-\lh\alpha, \phi \rh&+\int\limits_0^t{\int\limits{\-\LH\alpha \der{\phi}{t} +\LH\alpha\left(\der{\psi_1}{x}\der{\phi}{y}-\der{\psi_1}{y}\der{\phi}{x}\right)}}\\
&+\LH\psi_2\left(\der{\alpha}{x}\der{\phi}{y}-\der{\alpha}{y}\der{\phi}{x}\right)dxdydt=0,
\end{split}
\ee
for any function $\phi(x,y,t)$ which is smooth in $\D_0\times\inte$ and $\phi\Bigl|_{\partial\D_0}=0.$\\
We now put $\phi=\alpha$. We have proved that $\psi_1,\psi_2$ and $\alpha$ possess at least 2 weak derivatives which are essentially bounded. Therefore, the left hand side of Equation (\ref{eq:ap}) as a functional on $\phi$ is $W^{1,p}$-continuous and therefore we can substitute $\phi=\alpha$. We get
\be
-\lh \alpha,\alpha \rh +\int\limits_0^t{\int\limits_{\D_0}{-\LH\alpha\der{\alpha}{t} +
\LH\alpha\left(\der{\psi_1}{x}\der{\alpha}{y}-\der{\psi_1}{y}\der{\alpha}{x}\right)dxdydt}}=0.
\ee
We can integrate by parts:
\be
\begin{split}
\frac{1}{2}\|\alpha \|_{h}^2
&=\int\limits_0^t{\int\limits_{\D_0}{K(x,y)\dermix{\psi_1}{x}{y}\left[
\left(\der{\alpha}{x}\right)^2-\left(\der{\alpha}{y}\right)^2  \right] dxdy}dt}\\
&+\int\limits_0^t{\int\limits_{\D_0}{K(x,y)\left(\der[2]{\psi_1}{y}-\der[2]{\psi_1}{x} \right)\der{\alpha}{x}\der{\alpha}{y}dxdy}dt}.
\end{split}
\ee
Since  $\der{\alpha}{x},\der{\alpha}{y}$ are in $W^{1,p}$ and thus are bounded, then we can write
\be
\label{eq:derab}
\left|\nabla\alpha\right|^2\leq M^2.
\ee
Denote
\be
z^2(t)=\|\alpha \|_h^2.
\ee
Differentiating in $t$, we get:
\be
z\diff{z}{t}\leq\int\limits_{\D_0}{\left(\left|\dermix{\psi_1}{x}{y}\right|+\left|\der[2]{\psi_1}{y}-\der[2]{\psi_1}{x}\right|\right)|\nabla\alpha|^2dxdy}.
\ee
Using Equation (\ref{eq:derab}):
\be
z\diff{z}{t}\leq CM^{\epsilon}\int\limits_{\D_0}{\left(\Bigl|\dermix{\psi_1}{x}{y}\Bigl|+\Bigl|\der[2]{\psi_1}{y}-\der[2]{\psi_1}{x}\Bigl|\right)(\nabla\alpha)^{2-\epsilon}dxdy}.
\ee
We now apply H\"older inequality and Yudovich's elliptic estimate to get:
\be
z\diff{z}{t}\leq  CM^{\epsilon}\left(\left|\left|\dermix{\psi}{x}{y}\right|\right|_{L^{\frac{2}{\epsilon}}}+\left|\left|\der[2]{\psi_1}{y}-\der[2]{\psi_1}{x}\right|\right|_{L^{\frac{2}{\epsilon}}}\right)z^{2-\epsilon}\leq M^{\epsilon}C\frac{2}{\epsilon}M_1z^{2-\epsilon},
\ee
where  $M_1=\|\LH\psi_1\|_{L^{\infty}(\D_0\times \inte)}$.
We integrate
\be
z^{\epsilon}(t)-z^{\epsilon}(0)=(CM)^{\epsilon}2CM_1t.
\ee
\be
\label{eq:diffz}
\frac{z(t)}{CM}\leq \left(2CM_1t+\left[\frac{z(0)}{CM}\right]^{\epsilon} \right)^{1/\epsilon}
\ee
Choose $\eta>0$.
Denote
\be
K=\left\lceil\frac{T}{4CM_1} \right\rceil.
\ee
We will chose a sequence $\delta_1,\delta_2,...\delta_{K+1}$ in the following manner. Take $\delta_{K+1}=\eta$ and with $\delta_{i+1},..,\delta_{K+1}$ chosen, choose $\epsilon_i$ such that
\be
\left[ \frac{3}{4} \right]^{1/\epsilon_i}< \delta_{i+1}.
\ee
With this choice of $\epsilon_i$, choose $\delta_i$ such that
\be
\frac{\delta_i}{CM}\leq \frac{1}{4^{1/\epsilon_i}}.
\ee
Denote
\be
\tau=\frac{T}{4CM_1}
\ee
For this sequence of choices, the following holds: if $z\left([k-1]\tau\right)<\delta_i$ then $z\left(k\tau\right)<\delta_{i+1}$. This is because using $\epsilon_i$ we apply equation (\ref{eq:diffz}) starting from time $(k-1)\tau$ then
\be
\begin{split}
z(k\tau) & \leq \left(2CM_1\tau +\left[\frac{z([k-1]\tau)}{CM}\right]^{\epsilon_i} \right)^{1/\epsilon_i}\\
&\leq \left(\frac{1}{2}+\delta_i^{\epsilon_i}\right)^{1/\epsilon_i}\leq  \left(\frac{3}{4}\right)^{1/\epsilon_i}\leq \delta_{i+1}.
\end{split}
\ee
Therefore, if we have two initial conditions such that $\|\psi_{01}-\psi_{02}\|_{H^1_0}<\delta_1$ then $\|\psi_1(T)-\psi_2(T)\|_{H^1_0}<\eta$.
In particular $z(t)=0$ for $z(0)=0$, for every $t\in\inte$.
\end{proof}
We will now prove uniqueness in a larger class of functions, following the work of Kato \cite{Kpr}. We will need the following Sobolev-type inequality (see \cite[II.2.2]{Ga}).
\begin{lemma}
\label{lm:crSob}
Let $\phi$ in $H^1_0(D)$ then for $2<p<\infty$
\be
\| \phi\|_{L^p}\leq \left(\frac{p}{2\sqrt{2}}\right)^{\frac{p-2}{p}}\|\phi\|_{L^2}^{2/p}\|\phi\|_{H_0^1(D)}^{\frac{p-2}{p}}.
\ee

\end{lemma}
\begin{proposition}Let $\psi_1,\psi_2$ be a weak solution of the Euler Equations such that $\LH\psi_1\in L^{\infty}(\D_0\times\inte)$ and $\LH\psi_2\in L^{\infty}(\inte,L^2(\D_0))$ that have the same initial condition $\psi_0$ with $\LH\psi_0\in L^{\infty}(\D_0)$ , forcing $F\in L^{\infty}(\D_0\times[0,T])$ and $f=\der{F_x}{y}-\der{F_y}{x}\in L^{\infty}(\D_0\times[0,T])$ then
\be
\psi_1=\psi_2.
\ee
\end{proposition}
\begin{proof}
Continuing along the lines of the previous proof, form the difference
\be
\alpha=\psi_1-\psi_2.
\ee
Again, use Lemma \ref{lm:tm} for $0$ and $t$ and substract the weak formulations to get
\be
\begin{split}
-\lh\alpha, \phi \rh&+\int\limits_0^t{\int\limits{\-\LH\alpha \der{\phi}{t} +\LH\alpha\left(\der{\psi_1}{x}\der{\phi}{y}-\der{\psi_1}{y}\der{\phi}{x}\right)}}\\
&+\LH\psi_2\left(\der{\alpha}{x}\der{\phi}{y}-\der{\alpha}{y}\der{\phi}{x}\right)dxdydt=0,
\end{split}
\ee
for any function $\phi(x,y,t)$ which is smooth in $\D_0\times\inte$ and $\phi\Bigl|_{\partial\D_0}=0.$\\
We now put $\phi=\alpha$. We have proved that $\psi_1,\psi_2$ and $\alpha$ possess at least 2 weak derivatives that are in $L^2$. Therefore, the left hand side of Equation (\ref{eq:ap}) as a functional on $\phi$ is $H^1$-continuous and therefore we can substitute $\phi=\alpha$. We get
\be
-\lh \alpha,\alpha \rh +\int\limits_0^t{\int\limits_{\D_0}{-\LH\alpha\der{\alpha}{t} +
\LH\alpha\left(\der{\psi_1}{x}\der{\alpha}{y}-\der{\psi_1}{y}\der{\alpha}{x}\right)dxdydt}}=0.
\ee
We can integrate by parts:
\be
\begin{split}
\frac{1}{2}\|\alpha \|_{h}^2
&=\int\limits_0^t{\int\limits_{\D_0}{K(x,y)\dermix{\psi_1}{x}{y}\left[
\left(\der{\alpha}{x}\right)^2-\left(\der{\alpha}{y}\right)^2\right]  dxdy}dt}\\
&+\int\limits_0^t{\int\limits_{\D_0}{K(x,y)\left(\der[2]{\psi_1}{y}-\der[2]{\psi_1}{x} \right)\der{\alpha}{x}\der{\alpha}{y}dxdy}dt}.
\end{split}
\ee
Differentiating in $t$
\be
\begin{split}
\frac{1}{2}\diff{}{t}\|\alpha \|_{h}^2&=\int\limits_{\D_0}{K(x,y)\dermix{\psi_1}{x}{y}
\left[\left(\der{\alpha}{x}\right)^2-\left(\der{\alpha}{y}\right)^2\right]  dxdy}\\
&+\int\limits_{\D_0}{K(x,y)\left(\der[2]{\psi_1}{y}-\der[2]{\psi_1}{x} \right)\der{\alpha}{x}\der{\alpha}{y}dxdy}.
\end{split}
\ee
Since $\LH\psi_1\in L^{\infty}$ then we can apply elliptic estimates
\be
\left\|\dermix{\psi_1}{x_i}{x_j}\right\|_{L^p}\leq CpM,
\ee
where $M=\sup\limits_{t\in\inte}{\|\LH\psi_1\|_{L^{\infty}(\D)}}$.
We use H\"older inequallity to get
\be
\frac{1}{2}\diff{}{t}\|\alpha \|_{h}^2\leq CpM\|\nabla\alpha\|^2_{L^{2q}},
\ee
where $\frac{1}{p}+\frac{1}{q}=1$. Since the helical norm $\|\alpha\|_h$ is equivalent to $H_0^1$-norm $\|\nabla\alpha\|_{L^2}$ and since $\alpha|_{\partial\D_0}=0$ we can rewrite the previous inequallity as
\be
\frac{1}{2}\diff{}{t}\|\alpha \|_{h}^2\leq CpM\|\nabla\alpha\|^2_{L^{2q}},
\ee
We use Lemma \ref{lm:crSob} to conclude
\be
\|\nabla\alpha\|^2_{L^{2q}}\leq \left(\frac{q}{\sqrt{2}}\right)^{2(1-1/q)}\|\nabla\alpha\|_{L^2}^{2/q}\|\nabla\alpha\|_{H^1_0}^{2(1-1/q)}.
\ee
Observe that by applying elliptic estimates for $p=2$, we have
\be
\|\nabla\alpha\|_{H^1_0}\leq 2C\|\LH\psi_1-\LH\psi_2\|_{L^2(\D_0)},
\ee
which is bounded in time by the assumptions of the proposition.
Denote
\be
z=\|\nabla\alpha\|_{h}.
\ee
Collecting the estimates and redefining $C$, we have
\be
z\diff{z}{t}\leq C p \left(\frac{q}{\sqrt{2}}\right)^{2(1-1/q)} z^{2/q}=C p \left(\frac{p}{(p-1)\sqrt{2}}\right)^{2/p} z^{2-2/p},
\ee
\be
\frac{p}{2}\diff{}{t}z^{2/p}\leq Cp \left(\frac{p}{(p-1)\sqrt{2}}\right)^{2/p}.
\ee
Therefore,
\be
z(t)\leq \frac{p}{(p-1)\sqrt{2}}(2Ct)^{p/2}.
\ee
Taking $\tau\leq \frac{1}{4C}$ and $p\rightarrow \infty$ proves that $z(\tau)=0$ and by iteration, for all times.
\end{proof}
\section{Proof of Existence}
\label{sc:exi}
In this section, we will prove existence of the weak solutions. Our strategy is to smoothen the initial data and the forcing, apply classical methods to prove existence of the solutions for smoothened data and then apply compactness arguments to extract a subsequence converging to the weak solution. A crucial step in the argument is an application of Lemma \ref{lm:psidermix}, which allows the required control of the compactness.
\subsection{Construction of a sequence of smooth approximations}
\label{sbc:constr}
\begin{claim}
\label{cl:mol}
Let $C\in \R^n$ be a domain with twice differentiable boundary, with compact closure, let $f\in L^{\infty}(C)$. Then there exists a sequence $f_n\in C^{\infty}(C)$ such that $f_n(x)$ converges to $f(x)$ for almost every $x\in C$ and $\|f_n\|_{L^{\infty}}\leq \|f\|_{L^{\infty}}$.
\end{claim}
\begin{proof}
Let $\sigma:\R^n\rightarrow [0,\infty)$ be a standard mollification kernel (A positive, $C^{\infty}(\R^n)$ function, which is compactly supported on the unit ball and whose total integral is one) . And define
\be
\sigma_{\ep}(x)=\frac{1}{\ep^n}\sigma\left(\frac{x}{\ep}\right).
\ee
Define the set $C_\ep$ by
\be
C_{\ep}=\{x\in C|d(x,\partial C)\geq \ep \},
\ee
the set of all $x\in C$ whose distance from the boundary of $C$ is greater then $\ep$. Take a sequence of $\ep_n$ going to zero and define
\be
f_n=\sigma_{\ep_n}\ast (\chi_{C_{2\ep_n}}f),
\ee
where $\ast$ is the convolution in $\R^n$. The sequence $\{f_n\}$ obeys the requirements of the lemma.
\end{proof}
Observe that by employing Lebesgue dominated convergence, we conclude that $f_n$ converges to $f$ in every $L^p$ for $p<\infty$ and weak-* in $L^{\infty}$.
\par Let us now return to the weak problem as defined in Definition \ref{df:main}. Let $\o_0=\LH\psi_0$ be the initial vorticity and $f=\der{F_y}{x}-\der{F_x}{y}$ be the forcing of the vorticity. We choose the sequences $\o_0^n \in C^{\infty}(\D_0) $ and $f^n\in C^{\infty}(\D_0\times\inte)$, which obey the conditions of the Claim \ref{cl:mol}. Define $\OO_0^n=\frac{1}{\k}\vec{\t} \o_0^n$ and $\mathbf{f}^n=\frac{1}{\k}\vec{\t} f^n$, the vector field $\vec{\t}$ was defined in equation (\ref{eq:vect}). We know extend $\OO_0^n$ and $\mathbf{f}^n$ to helical vector fields on $\D$ by Equation \ref{eq:recvec}.
\par We now obtain a sequence smooth initial velocity fields $\u_0^n$ and body forcing terms $\mathbf{F}^n$, which obey the orthogonality condition.
We now wish to solve the Euler Equations in the vorticity formulation with initial data $\u_0^n$ and vorticity forcing term $\mathbf{F}^n$ in the domain $\D$ .
We will employ the following theorem of Ferrari\footnote{We restate the theorem adapted to our setting in a standard manner.} (see \cite[Theorem 2]{Fer}).
\begin{theorem}[Ferrari]
\label{th:Fer}
Let $\OO_0\in H^s(D),s\geq 2$ be the initial vorticity , $D$ is bounded simply connected domain then there exists a solution $\OO\in C([0,T],H^s(D))$ if and only if
\be
\int\limits_{0}^T {\|\o \|_{L^\infty(D)} dt}\text{ is finite}.
\ee
\end{theorem}
Theorem \ref{th:Fer} is an extension of the theorem by Beal, Kato and Majda \cite{BKM}, who proved the same criterion for Euler equations in the whole $\R^n$.
We will need a modification of Theorem \ref{th:Fer}, for a helical domain $\D$, which is bounded in the $\hat{x},\hat{y}$ direction and periodic in the $\hat{z}$ direction and solution, which are periodic in the $\hat{z}$ direction. It is proved along the same lines, as the original theorem of Ferrari. We now apply the theorem of Ferrari.
\begin{lemma}
\label{lm:oninftest}
For every $n\in\mathbf{N}$ there exists a smooth helical solution of the Euler Equations  $\OO^n\in C^{\infty}(\D\times\inte)$ with the initial vorticity $\OO^n_0$ and forcing $\mathbf{f}_n$. Moreover, it obeys the following estimate:
\be
\sup\limits_{t\in\inte}{\|\OO^n(t) \|_{L^{\infty}}}\leq M,
\ee
where
\be
M=C(\|\OO_0\|_{L^{\infty}}+\int\limits_{0}^T{\|\mathbf{f}(t)\|_{L^\infty}}).
\ee
\end{lemma}
\begin{proof} Observe that if $\OO^n$ obeys the required estimate then it fulfills the conditions of the theorem of Ferrari and it is in the class $H^s$. Then we take $s$ to infinity will prove that $\OO$ is a $C^{\infty}$ vector field.
Observe that $\OO$ are classical differentiable functions and therefore, we can apply the standard tools of calculus. We will use Corollary \ref{cr:transo} to conclude the estimate for $\O_z$. Then we will use Lemma \ref{lm:repo} and the fact that $\D$ is contained in a cylinder of finite radius and therefore $|\O_x|$ and $|\O_y|$ are smaller then $C|\O_z|$ for some bounded constant $C$.
\end{proof}
From now on we will fix $p\geq 2$. Unlike in the previous section, where the dependence of the constants on $p$ was crucial to the argument, in this section it will not play any role, since $p$ is fixed and we will subsume this dependence into constants.
\par After obtaining the solutions $\OO^n(\x,t)$, we define
\be
\o^n=(\OO^n)_z,
\ee
a sequence of smooth helical functions, which we reduce to function on $\D_0$. Then we obtain solutions to the problem
\be
\LH\psi^n=\o^n,\quad \psi^n|_{\partial\D_0}=0.
\ee
\subsection{Uniform bounds on approximating sequence}
\begin{claim}
\label{cl:psiw2p} Let $\psi^n$ be the stream function corresponding to $\o^n$ then $\psi^n$ obeys the following estimate:
\be
\sup\|\psi^n(t)\|_{W^{2,p}(\D_0)}\leq CM.
\ee
\end{claim}
\begin{proof}
This is the immediate corollary of the elliptic regularity, Corrolary \ref{cr:YEL} together with the fact that the $L^{\infty}$ norm of the vorticity is uniformly bounded.
\end{proof}
We will also need a uniform bound on the time derivative of $\psi^n$:
\begin{lemma}
\label{lm:derpsin}
Let $\psi^n$ be the stream function corresponding to $\o^n$
\be
\sup\limits_{t\in\inte} \|\der{\psi^n}{t}\|_{W^{1,p}(\D_0)}\leq C(\|f\|_{L^{\infty}}+M^2).
\ee
\end{lemma}
\begin{proof}
We will apply Lemma \ref{lm:psidermix}. The estimate in Lemma \ref{lm:psidermix} depends on the norm of the body forcing function $\mathbf{F}^n$. We wish to replace it with the estimate on $\mathbf{f}^n$. Since the velocity fields are incompressible, we can replace $\mathbf{F}^n$ with it's incompressible part without affecting the dynamics. But then $\mathbf{F}^n$ obeys
\be
\Delta\mathbf{F}^n=\nabla\wedge \mathbf{f}^n.
\ee
We now apply the elliptic regularity in the form of Theorem \ref{th:negEE} to conclude that the norms of the body forcing term are controlled by the vorticity forcing term. Then we can apply Lemma \ref{lm:psidermix} and the uniform estimate on $\|\psi^n(t)\|_{W^{2,p}}$ from Claim \ref{cl:psiw2p} to conclude the lemma.
\end{proof}
\subsection{Extraction of a solution to the weak problem}
After establishing uniform in $n$ estimates, we will now apply Aubin compactness argument for the sequence $\{\psi^n\}$. We will use the following theorem, \cite[Lemma 8.4]{CF}
\begin{theorem}
\label{th:comp}
Let $X_1,X_0,X_{-1}$ be reflective Banach spaces such that $X_1$ is compactly embedded into $X_0$ and $X_{0}$ is continuous embedded in $X_{-1}$. Let $u_m$ be a bounded sequence in $L^{p_1}(\inte,X_1)$ such that $\diff{u_m}{t}$ is bounded in $L^{p_2}(\inte,X_{-1})$ then there exists subsequence $u_{m'}$ which converges in $L^{p_1}(\inte,X_0)$.
\end{theorem}
\begin{corr}
There exist a subsequence of $\psi^n$ which converges in $$L^p(\inte,W^{1,p}(\D_0)).$$
\end{corr}
\begin{proof}
We will apply Theorem \ref{th:comp} with $X_1=W^{2,p}(\D_0)$ and $X_0=X_{-1}=W^{1,p}(\D_0)$. The uniform estimates are Claim \ref{cl:psiw2p} and Lemma \ref{lm:derpsin}. Observe that we obtained estimates in $L^\infty(\inte,X_1)$ for the functions and $L^\infty(\inte,X_{-1})$ for the time derivatives, clearly they imply $L^p(\inte,X_1)$ and $L^p(\inte,X_{-1})$ estimates, respectively.
\end{proof}
\par
\begin{proposition} Let $\psi_0$ be such that $\LH\psi_0=\o_o\in L^{\infty}(\D_0)$ and $\mathbf{F}\in L^{\infty}(\D_0\times\inte)$ then there exists $\psi$ such that $\LH\psi=\o\in L^{\infty}(\D_0\times\inte)$ which obeys the identity
\be
\begin{split}
\int\limits _{\D_0} {\o_0(x) \phi(x,0) dxdy}&-\int\limits_{0}^T{\int\limits_{\D_0}{\o(x,t)\der{\phi}{t}(x,t)dxdy}dt}+\int\limits_{0}^T{\int\limits_{\D_0}{\der{\psi}{y}\o \der{\phi}{x}dxdy}dt}\\
&-\int\limits_{0}^T{\int\limits_{\D_0}{\der{\psi}{x}\o\der{\phi}{y}dxdy}dt}=\int\limits_{0}^T{\int\limits_{\D_0}{\left(\der{F_y}{x}-\der{F_x}{y} \right)\phi dxdy}dt},
\end{split}
\ee
for every test function $\phi$.
\end{proposition}
\begin{proof}
Let $\o^n=\O^n_z$ be the sequence of functions, which we built in Subsection \ref{sbc:constr}. By Lemma \ref{lm:oninftest}, they are uniformly bounded in $L^{\infty}(\D_0)$ and therefore, we can extract a subsequence, which converges in weak-* topology to a function $\o$. By extracting from the subsequence a subsubsequence, we see that $\psi^n$ will converge in $L^p(\inte,W^{1,p}(\D_0))$ to a function $\psi$. By uniqueness of the limits, $\psi$ must be the stream function which corresponds to $\o$. Therefore $\nabla \psi^n$ will converge to $\nabla\psi$ in $L^p(\inte,L^p(D_0))\subseteq L^p(\D_0\times\inte)$. Therefore $\o^n \der{\psi^n}{x_i}$ converge weakly in $L^p(\D_0\times\inte)$. For any smooth function$\phi$ with compact support in $\D_0\times\inte$,the functions $\psi_n$ is the weak solution of the Euler Equations in the sense of Definition \ref{df:main} and therefore they obey
\be
\begin{split}
\int\limits _{\D_0} {\o_0^n(x) \phi(x,0) dxdy}&-\int\limits_{0}^T{\int\limits_{\D_0}{\o^n\der{\phi}{t}(x,t)dxdy}dt}+\int\limits_{0}^T{\int\limits_{\D_0}{\der{\psi^n}{y}\o^n \der{\phi}{x}dxdy}dt}\\
&-\int\limits_{0}^T{\int\limits_{\D_0}{\der{\psi^n}{x}\o^n\der{\phi}{y}dxdy}dt}=\int\limits_{0}^T{\int\limits_{\D_0}{\left(\der{F^n_y}{x}-\der{F^n_x}{y} \right)\phi dxdy}dt}.
\end{split}
\ee
We now see that all the terms in this equation converge $L^p$ weakly to the appropriate limits and therefore, $\psi$, the limit, will obey the required equality and therefore it is the solution of the Euler Equations.
\end{proof}
\section*{Acknowledgments}
This work was supported in part by the BSF
grant no. 2004271, the ISF grant no. 120/06, and the
NSF grants no. DMS-0504619 and no. DMS-0708832.
\bibliographystyle{plain}
\bibliography{Bib-HelicalEuler}
\end{document}